\documentclass[11pt,thmsa]{article}%
\usepackage{amsmath}
\usepackage{amsfonts}
\usepackage{amssymb}
\usepackage{graphicx}%
\newtheorem{theorem}{Theorem}[section]

\newtheorem{corollary}[theorem]{Corollary}

\newtheorem{lemma}[theorem]{Lemma}

\newtheorem{problem}[theorem]{Problem}
\newtheorem{proposition}[theorem]{Proposition}

\newenvironment{proof}[1][Proof]{\noindent\textbf{#1.} }{\ \rule{0.5em}{0.5em}}

\begin{document}
\title{Even dimensional homogeneous Finsler spaces with positive flag curvature\thanks{Supported by NSFC (no. 11271216, 11271198, 11221091), State Scholarship Fund of CSC (no. 201408120020), SRFDP of China, Science and Technology Development Fund for Universities and Colleges in Tianjin (no. 20141005), and Doctor fund of Tianjin Normal University (no. 52XB1305), }}
\author{Ming Xu$^1$, Shaoqiang Deng$^2$, Libing Huang$^2$\thanks{Corresponding author. E-mail: Huanglb@nankai.edu.cn} and Zhiguang Hu$^1$\\
\\
$^1$College of Mathematics\\
Tianjin Normal University\\
 Tianjin 300387, P.R. China\\
 \\
$^2$School of Mathematical Sciences and LPMC\\
Nankai University\\
Tianjin 300071, P.R. China}
\date{}
\maketitle

\begin{abstract}
In this paper, we use the technique of Finslerian submersion to deduce a flag curvature formula for homogeneous Finsler spaces. Based on this  formula, we give a complete classification of even-dimensional smooth
coset spaces $G/H$ admitting
$G$-invariant Finsler metrics with positive flag curvature. It turns out that the classification list  coincides with that of the even dimensional homogeneous Riemannian  manifolds with positive sectional curvature obtained by N.R. Wallach. We also find out all the coset spaces admitting invariant non-Riemannian Finsler metrics with positive flag curvature.



\end{abstract}

\section{Introduction}
One of the central problems in Riemannian geometry is to classify compact smooth manifolds admitting
Riemannian metrics of positive sectional curvature, Ricci curvature or scalar curvature.
For Ricci or scalar curvature, the study on this problem is very fruitful. However, for sectional curvature
it is extremely involved.
In fact, up to now  only  homogeneous manifolds $G/H$ which admit  $G$-invariant Riemannian metrics with positive sectional curvature have been completely classified, up to local isometries; see \cite{AW75, BB76, Ber61, Wallach1972, XuWolf2015}. In the inhomogeneous case,  only a few examples and possible candidates have been found among
compact biquotients and cohomogeneity one spaces; see for example
\cite{BA96, Der11, ES82, GWZ08, GVZ12}.

In  Finsler geometry, in particular in the homogeneous settings, the corresponding problem can be stated as  the following.
\begin{problem} \label{main-problem}
Classify all smooth coset spaces admitting invariant Finsler metrics of positive flag curvature, up to local isometries.
\end{problem}
For simplicity, we say that a homogeneous space is {\it positively curved} if it admits a homogeneous Finsler metric with positive flag curvature,
or it has already been endowed with such a metric. By the theorem of Bonnet-Myers for Finsler spaces, a positively curved homogeneous space must be compact.

Problem \ref{main-problem}
possesses the same importance as the one in  Riemannian geometry and it seems more difficult.
Nevertheless,  significant progress has been made in some special cases. For example,
the second and the fourth authors of this paper proved in \cite{HuDeng2013} that if $G$ is a connected simply connected
Lie group which admits a left invariant positively curved Finsler metric, then it must be isomorphic to $\mathrm{SU}(2)$. Moreover,  they found the same rank inequality for positively curved homogeneous Finsler space as in  Riemannian geometry. They also presented a complete classification of homogeneous Randers metrics with positive flag curvature and zero S-curvature in \cite{HD11}, which was later generalized by the first and the second authors to homogeneous $(\alpha,\beta)$-metrics \cite{XuDeng2015}.

Equally significant achievement has been made while the first and the second author classified positively curved normal homogeneous Finsler spaces \cite{XuDeng2015-02}, generalizing the work \cite{Ber61} of M.~Berger. In \cite{XuDeng2015-02}, the authors set up a general theme for the study of Problem \ref{main-problem}, and established some useful algebraic techniques by  introducing the method of Finslerian submersion to the study of homogeneous Finsler spaces.

The purpose of this paper is to study Problem \ref{main-problem} in the even
dimensional case. To reduce Problem \ref{main-problem} to a purely algebraic problem, a flag
curvature formula for homogeneous Finsler spaces is needed. The third author
of this work has set up a systematical machinery to rewrite curvatures of
homogeneous Finsler spaces with invariant frames \cite{Huang2013}. In this work, we will
provide a more fundamental technique,
using Finslerian submersion to deduce the
wanted flag curvature formula.


The main result of this paper is the following.

\begin{theorem}\label{thm1}
 Let $G$ be a compact connected simply connected Lie group and $H$  a connected closed subgroup such that the dimension of the coset space $G/H$ is even.
Suppose that there exists a $G$-invariant Finsler metric  on $G/H$  with positive flag curvature. Then
there exists a $G$-invariant Riemannian metric on $G/H$ with positive sectional curvature.
\end{theorem}

Combined with the work \cite{Wallach1972} of N.R.~Wallach, Theorem \ref{thm1} gives the following corollary.
\begin{theorem}\label{thm2}
Let $G$ be a compact connected simply connected Lie group and $H$ a connected closed subgroup of $G$ such that the dimension of the coset space $G/H$ is even. Assume $\mathfrak{h}=\mathrm{Lie}(H)$ does not contain any nonzero ideal of $\mathfrak{g}=\mathrm{Lie}(G)$. If $G/H$ admits a $G$-invariant Finsler metric with positive flag curvature, then the pair $(G,H)$ must be one of the following:
\begin{enumerate}
\item Rank one symmetric pairs of compact type. In this case, $G/H$ is one of the even dimensional spheres, the complex projective spaces, the quaternion projective spaces or the $16$-dimensional Cayley plane. Moreover, any $G$-invariant Finsler metric on $G/H$ must be Riemannian and of positive sectional curvature.
\item The pair $(G_2, \mathrm{SU}(3))$. In this case, the coset space
is $S^6=\mathrm{G}_2/\mathrm{SU}(3)$ on which any $G$-invariant Finsler metric must be a Riemannian metric of constant positive sectional curvature.
\item The pair $(\mathrm{Sp}(n), \mathrm{Sp}(n-1)\mathrm{U}(1))$. In this case, the coset space is $\mathbb{C}\mathrm{P}^{2n-1}=\mathrm{Sp}(n)/ \mathrm{Sp}(n-1)\mathrm{U}(1)$,  on which there exist  $G$-invariant Riemannian metrics as well as non-Riemannian Finsler metrics  with positive flag (sectional) curvature.
\item The pair $(\mathrm{SU}(3), T^2)$, where $T^2$ is a maximal torus of $\mathrm{SU}(3)$.
In this case, on the coset space $F^6=\mathrm{SU}(3)/T^2$ there exist $G$-invariant Riemannian metrics as well as non-Riemannian Finsler metrics with positive flag (sectional) curvature.
\item The pair $(\mathrm{Sp}(3), \mathrm{Sp}(1)\times \mathrm{Sp}(1)\times \mathrm{Sp}(1))$. In this case, on the coset space $$F^{12} =\mathrm{Sp}(3) / \mathrm{Sp}(1)\times \mathrm{Sp}(1)\times \mathrm{Sp}(1)$$ there exist invariant Riemannian metrics as well as non-Riemannian Finsler metrics with positive flag (sectional) curvature.
\item The pair $(\mathrm{F}_4,\mathrm{Spin}(8))$. In this case, on the coset space $F^{24}=\mathrm{F}_4/\mathrm{Spin}(8)$
there exist $G$-invariant Riemannian metrics as well as non-Riemannian metrics with positive
flag (sectional) curvature.
\end{enumerate}
\end{theorem}


Essentially Theorem \ref{thm2}
gives the solution for Problem \ref{main-problem} in the even dimensional case, i.e., a complete local description of even dimensional smooth coset spaces admitting homogeneous Finsler metric with positive flag curvature (see the final remark at the end).
However, we must
point out that, for some cases of the list in Theorem \ref{thm2}, it is very hard to find the exact conditions under which an invariant Finsler metric is positively curved. In fact, this problem is even very complicated in the Riemannian case; see \cite{VZ09} for the detailed discussion for the parameters of positively curved homogeneous Riemannian metrics on spheres. For non-Riemannian homogeneous Randers metrics with vanishing S-curvature considered in \cite{HD11},
all positively curved Finsler metrics can be exactly and explicitly presented, but the method cannot be applied   to the even dimensional case in this work.

In Section 2, we recall some fundamental definitions and known results needed in this paper. In Section 3, we introduce the Finslerian submersion and apply this method to the study of homogeneous Finsler spaces. In Section 4, we use the technique of Finslerian submersion to prove a flag curvature formula for
homogeneous Finsler spaces. An   intrinsic proof of the same formula is also given in this ection. In Sections 5 and  6 we prove Theorem \ref{thm1} and Theorem \ref{thm2}, respectively.

\section{Preliminaries in Finsler geometry}

In this section we recall some definitions and results on Finsler spaces. For general Finsler spaces we refer the readers to \cite{BCS00} and \cite{CS04}; for homogeneous Finsler spaces we refer to \cite{DE12}.
\subsection{Minkowski norms and Finsler metrics}

A {\it Minkowski norm} on a real vector space $\mathbf{V}$, $\dim \mathbf{V}=n$, is a continuous real-valued function
$F:\mathbf{V}\rightarrow [0,+\infty)$ satisfying the following conditions:
\begin{enumerate}
\item $F$ is positive and smooth on $\mathbf{V}\backslash\{0\}$;
\item $F(\lambda y)=\lambda F(y)$ for any $\lambda\geq 0$;
\item with respect to any linear coordinates $y=y^i e_i$, the Hessian matrix
\begin{equation}
(g_{ij}(y))=\left(\frac12[F^2]_{y^i y^j}(y)\right)
\end{equation}
is positive definite at any nonzero $y$.
\end{enumerate}

The Hessian matrix $(g_{ij}(y))$ and its inverse $(g^{ij}(y))$ can be used   to raise and lower down indices of relevant tensors in Finsler geometry.

 For any $y\neq 0$, the Hessian matrix $(g_{ij}(y))$ defines an inner product $\langle\cdot,\cdot\rangle_y$ on $\mathrm{V}$ by  $$\langle u,v\rangle_y=g_{ij}(y)u^i v^j,$$
where $u=u^i e_i$ and $v=v^i e_i$. Sometimes we denote the above inner product as $\langle\cdot,\cdot\rangle^F_y$ if there are several norms in consideration.
 This inner product can also be
expressed as
\begin{equation}\label{formula-2-2}
\langle u,v\rangle_y=\frac{1}{2}\frac{\partial^2}{\partial s\partial t}[F^2(y+su+tv)]|_{s=t=0}.
\end{equation}
It is easy to check that the above definition  is independent of the choice of linear coordinates.

 Let $M$ be a smooth manifold  with dimension $n$. A Finsler metric $F$ on $M$ is a continuous function $F:TM\rightarrow [0,+\infty)$
such that it is positive and smooth on the slit tangent bundle $TM\backslash 0$, and its restriction in each tangent
space is a Minkowski norm. We will call $(M,F)$ a {\it Finsler manifold} or a {\it Finsler space}.

Here are some important examples.

Riemannian metrics are a special class of Finsler metrics such that the Hessian matrices only
depend on $x\in M$ and not on $y\in T_x M$. For a Riemannian manifold, the metric is often referred to as a
global smooth section $g_{ij}(x)dx^i dx^j$ of $\mathrm{Sym}^2(T^*M)$.

Randers metrics are the most simple and important class of non-Riemannian metrics in Finsler geometry. They are defined as
$F=\alpha+\beta$, in which $\alpha$ is a Riemannian metric and $\beta$ is a 1-form (see \cite{RA41}). The notion of Randers metrics has been  naturally
generalized to $(\alpha,\beta)$-metrics. An $(\alpha,\beta)$-metric is a Finsler metric of    the form $F=\alpha\phi(\beta/\alpha)$, where $\phi$ is   a positive smooth
real function,  $\alpha$  is a Riemannian metric and  $\beta$ is a $1$-form. In recent years, there have been a lot of research works
concerning $(\alpha,\beta)$-metrics as well as  Randers metrics.

Recently, the first two authors of this paper  defined and studied $(\alpha_1,\alpha_2)$-metrics in \cite{XuDeng2014},  generalizing
$(\alpha,\beta)$-metrics. Using the same idea,
we can define $(\alpha_1,\alpha_2,\ldots,\alpha_k)$-metrics. Let $\alpha$ be a Riemannian metric on $M$, such
that $TM$ can be $\alpha$-orthogonally decomposed as $TM=\mathcal{V}_1\oplus\cdots\oplus\mathcal{V}_k$, in
which each $\mathcal{V}_i$ is an $n_i$-dimensional linear sub-bundle with $n_i>0$ respectively. Let $\alpha_i$
be the restriction of $\alpha$ to each $\mathcal{V}_i$ and naturally regarded as functions on $TM$ with
$\alpha^2=\alpha_1^2+\cdots+\alpha_k^2$. Then a metric $F$ is called an $(\alpha_1,\alpha_2,\ldots,\alpha_k)$-metric
if $F$ can be presented as $F=\sqrt{L(\alpha_1^2,\cdots,\alpha_k^2)}$, where $L$ is a positive smooth real function on an open subset of $\mathbb{R}^k$ satisfying certain conditions. In the following we will show that on some coset spaces there exists (non-Riemannian) invariant $(\alpha_1,\alpha_2)$ or $(\alpha_1,\alpha_2,\alpha_3)$-metrics with positive flag curvature. In our opinion, these metrics will be of great interest in Finsler geometry.

\subsection{Geodesic spray and geodesic}

On a Finsler space $(M,F)$, a local coordinate system $\{x=(x^i)\in M, y=y^j\partial_{x^j}\in T_x M\}$ will be called
a {\it standard} local coordinate system.

The geodesic spray is a vector field $G$ defined on $TM\backslash 0$. In any standard local coordinate system, it can be
presented as
\begin{equation}
G=y^i\partial_{x^i}-2G^i\partial_{y^i},
\end{equation}
in which
\begin{equation}
G^i=\frac{1}{4}g^{il}([F^2]_{x^k y^l}y^k-[F^2]_{x^l}).
\end{equation}

A non-constant curve $c(t)$ on $M$ is called a geodesic if $(c(t),\dot{c}(t))$ is an integration curve of $G$,
in which the tangent field $\dot{c}(t)=\frac{d}{dt}c(t)$ along the curve gives the speed. For any standard local
coordinates, a geodesic $c(t)=(c^i(t))$ satisfies the equations
\begin{equation}
\ddot{c}^{i}(t)+2G^i(c(t),\dot{c}(t))=0.
\end{equation}

It is well known that $F(\dot{c}(t))\equiv\mathrm{const}$, i.e., the geodesics we are considering are
geodesics of nonzero constant speed.

\subsection{Riemann curvature and flag curvature}
On a Finsler space, we have a similar {\it Riemann curvature} as in Riemannian geometry. It can be defined either by  Jacobi fields or the structure equation for
the curvature of the Chern connection.

For a standard local coordinates system, the Riemann curvature is a linear map
$R_y=R^i_k(y)\partial_{x^i}\otimes dx^k:T_x M\rightarrow T_x M$, defined by
\begin{equation}\label{formula-2-6}
R^i_k(y)=2\partial_{x^k}G^i-y^j\partial^2_{x^j y^k}G^i+2G^j\partial^2_{y^j y^k}G^i-\partial_{y^j}G^i\partial_{y^k}G^j.
\end{equation}
The Riemann curvature $R_y$ is self-adjoint with respect to $\langle\cdot,\cdot\rangle_y$.

Using the Riemann curvature, we can generalize the notion of  sectional curvature in Riemannian geometry to {\it flag curvature} in Finsler geometry. Let
$y$ be a nonzero tangent vector in $T_xM$ and $\mathbf{P}$ a tangent plane in $T_x M$ containing $y$, linearly spanned
by $y$ and $v$ for example. Then the flag curvature of the flag $(\mathbf{P}, y)$ is given by
\begin{equation}\label{formula-2-7}
K(x,y,y\wedge v)=K(x,y,\mathbf{P})=\frac{\langle R_y v,v\rangle_y}{\langle y,y\rangle_y\langle v,v\rangle_y-\langle y,v\rangle_y^2}.
\end{equation}
Fixing $x\in M$,
the flag curvature in (\ref{formula-2-7}) does not  depend on the choice of $v$ but on $\mathbf{P}$ and $y$. When $F$ is a Riemannian metric, it is in fact
the sectional curvature, which depends on $\mathbf{P}$ only.
 Sometimes we will write the flag curvature of a Finsler metric $F$ as $K^F(x,y,y\wedge v)$ or $K^F(x,y,\mathbf{P})$ to indicate the metric explicitly.

Z.~Shen has made the following important observation which relates the Riemann curvature of a Finsler metric to that of a Riemannian metric.

Let $Y$ be a tangent vector field on an open set $\mathcal{U}\subset M$ which is nowhere vanishing. Then the Hessian matrices
$(g_{ij}(Y(x)))$ with respect to a  standard local coordinate system
define a smooth Riemannian metric on $\mathcal{U}$ which is  independent of  the local coordinates system. We denote this Riemannian metric as $g_Y$, and call it the {\it localization} of $F$ at $Y$. The Riemann
curvatures for $F$ and $g_Y$ is denoted as $R^F_y$ and $R^{g_Y}_y$, respectively.

If $Y$ is a nonzero geodesic field on an open set $\mathcal{U}\subset M$, i.e., if each integration curve of $Y$ is a geodesic of nonzero constant speed,
then we have  the following theorem of Z.~Shen.

\begin{theorem} \label{theorem-2-1}
Let $Y$ be a geodesic field on an open set $\mathcal{U}\subset M$ such that for $x\in \mathcal{U}$,  $y=Y(x)\neq 0$. Then $R^F_y=R^{g_Y}_y$.
\end{theorem}

It follows immediately from the definition of flag curvature and Theorem \ref{theorem-2-1} that if $\mathbf{P}$ is a tangent plane in $T_x M$ containing $y$, then $K^F(x,y,\mathbf{P})={K}^{g_Y}(x,\mathbf{P})$.

\section{Submersion of  Homogeneous Finsler spaces}

In this section we recall some definitions and results on Finslerian submersion and homogeneous Finsler spaces, and explore the relationship
between these  subjects.

\subsection{Submersion and the subduced metric}

A linear map $\pi:(\mathbf{V}_1,F_1)\rightarrow(\mathbf{V}_2,F_2)$ between two Minkowski spaces is called an {\it isometric submersion}
(or simply {\it submersion}), if it maps the unit ball $\{y\in\mathbf{V}_1|F_1(y)\leq 1\}$ in $\mathbf{V}_1$ onto the unit ball
$\{y\in \mathbf{V}_2|F_2(y)\leq 1\}$ in $\mathbf{V}_2$. Obviously a submersion map $\pi$ must be surjective. The Minkowski norm
$F_2$ on $\mathbf{V}_2$ is uniquely determined by the submersion by the following equality,
\begin{equation}
F_2(w)=\inf\{F_1(v)|\pi(v)=w\}.\label{formula-3-8}
\end{equation}
Given the Minkowski space $(\mathbf{V}_1,F_1)$ and the surjective linear map $\pi:\mathbf{V}_1\rightarrow\mathbf{V}_2$, there
exists a unique Minkowski norm $F_2$ on $\mathbf{V}_2$ such that $\pi$ is a submersion. We will call $F_2$ the subduced norm. For  details of submersion of Finsler metrics, we refer the readers to \cite{PD01}.

To clarify the relationship between the Hessian matrices of $F_1$ and $F_2$, we need the notion of horizonal lift. Given a nonzero vector
$w$ in $\mathbf{V}_2$, the infimum in (\ref{formula-3-8}) can be reached by a unique vector $v$, which is called the horizonal lift of $w$ with respect to the
submersion $\pi$. The horizonal lift $v$ can also be determined by
\begin{equation}
\langle v,\mathrm{ker}\pi\rangle^{F_1}_v=0, \mbox{ and }\pi(v)=w.
\end{equation}
Then the  Hessian matrix of $F_2$ at $w$ is determined by the following proposition.

\begin{proposition}\label{proposition-3-1}
Let $\pi:(\mathbf{V}_1,F_1)\rightarrow(\mathbf{V}_2,F_2)$ be a submersion between Minkowski spaces. Assume that $v$ is the horizonal lift of a
nonzero vector $w$ in $\mathbf{V}_2$. Then the map $\pi:(\mathbf{V}_1,\langle\cdot,\cdot\rangle^{F_1}_v)
\rightarrow(\mathbf{V}_2,\langle\cdot,\cdot\rangle^{F_2}_w)$
is a submersion between Euclidean spaces.
\end{proposition}

For any vector $u\in\mathbf{V}_1\backslash\mathrm{ker}\pi$, the horizonal lift of $\pi(u)$ is called a horizonal shift of $u$.
 If $w=0\in\mathbf{V}_2$ or $u\in\ker \pi$,  then we naturally choose $v=0$ as its horizonal lift or horizonal shift respectively.
The map from vector $u\in\mathbf{V}_2$ ($u\in\mathbf{V}_1$ resp.)to
its horizonal lift (horizonal shift resp.) is smooth when $u\neq 0$ ($u\notin\mathrm{ker}\pi$ resp.).

A smooth map $\rho:(M_1,F_1)\rightarrow (M_2,F_2)$ between two Finsler spaces is called a submersion, if for any $x\in M_1$, the induced tangent map
$\rho_*:(T_x M_1,F_1(x,\cdot))\rightarrow(T_{\rho(x)}M_2,F_2(\rho(x),\cdot))$ is a submersion between Minkowski spaces. Restricted to the image of
the submersion $\rho$, the metric $F_2$ is uniquely determined by $F_1$ and the submersion. Let $F_1$ be a Finsler metric on $M_1$ and $\rho:M_1\rightarrow M_2$ be a surjective smooth map. If there is a metric $F_2$ on $M_2$ which makes $\rho$ a submersion, then we call  $F_2$ the subduced metric from $F_1$ and
$\rho$. Note that the subdued metric may not exist; but when it does, it must be unique.

For a submersion between Finsler spaces, the horizonal lift (or the horizonal shift) of a tangent vector field can be similarly defined. The corresponding
integration curves define the horizonal lift (or shift) of smooth curves. Horizonal lift provides a one-to-one correspondence between the geodesics on $M_2$ and
the horizonal geodesics on $M_1$, so the horizonal lift of a geodesic field is also a geodesic field. Using Theorem \ref{theorem-2-1}, Proposition
\ref{proposition-3-1} and the curvature formula of Riemannian submersions, one can prove  the following theorem (see \cite{PD01}).

\begin{theorem}\label{theorem-3-2}
Let $\rho:(M_1,F_1)\rightarrow(M_2,F_2)$ be a submersion of Finsler spaces. Assume that $x_2=\rho(x_1)$,  and that $y_2,v_2\in T_{x_2}M_2$  are two linearly independent tangent vectors. Let $y_1$
be the horizonal lift of $y_2$, and $v_1$ the horizonal lift of $v_2$ with respect to the induced submersion $\rho_*:(T_{x_1}M_1,\langle\cdot,\cdot\rangle_{y_1})
\rightarrow(T_{x_2}M_2,\langle\cdot,\cdot\rangle_{y_2})$. Then the flag curvature of $(M_1,F_1)$ and $(M_2, F_2)$ satisfies  the following inequality:
\begin{equation}
K^{F_1}(x_1,y_1,y_1\wedge v_1)\leq K^{F_2}(x_2,y_2,y_2\wedge v_2).
\end{equation}
\end{theorem}

\subsection{Homogeneous Finsler metrics subduced by a submersion}

A connected Finsler space $(M,F)$ is called homogeneous, if
the full group of  isometries  $I(M,F)$ acts transitively on $M$. In this case, the identity component $I_0(M, F)$ also acts transitively on $M$ (see \cite{HE78}). If $G$ is a
 closed connected subgroup  of $I_0(M,F)$  which acts transitively on $M$,
then $M$ can
also be identify with the coset space  $G/H$, in which $H$ is the isotropy subgroup at some
point. In this case, we usually say that  $F$ is a $G$-homogeneous Finsler metric on $M$. Notice that the isometry group of any compact Finsler space
must be compact.  Hence  we only need to consider compact connected $G$ with
a closed subgroup $H$, endowed with  a $G$-homogeneous Finsler metric
$F$ on $G/H$. Denote the natural projection from $G$ to $G/H$ as $\pi$.
Then $H$ can be chosen as  the isotropy subgroup at $o=\pi(e)$.

Let $\mathfrak{g}$ and $\mathfrak{h}$ be the Lie algebras of $G$ and $H$ respectively. Because $H$ is a compact subgroup of $G$, we can find
an $\mathrm{Ad}(H)$-invariant inner product $\langle\cdot,\cdot\rangle$ on
$\mathfrak{g}$. Then we have an {\it $\mathrm{Ad(H)}$-invariant decomposition} for
the homogeneous space $G/H$,
\begin{equation}
\mathfrak{g}=\mathfrak{h}+\mathfrak{m},
\end{equation}
in which $\mathfrak{m}$ is the orthogonal complement of $\mathfrak{h}$.
Obviously this decomposition is reductive, i.e. $[\mathfrak{h},\mathfrak{m}]
\subset\mathfrak{m}$. The vector space $\mathfrak{m}$ can be
identified with the tangent space $T_o(G/H)$.


In this paper, we will mainly deal with the situations that the homogeneous Finsler space $(M,F)$ is compact (from the property of being positively curved), with
$M=G/H$ in which $G$ is a closed subgroup of the compact transformation group $I_0(M,F)$. In this case, we can choose an $\mathrm{Ad}(G)$-invariant inner product on $\mathfrak{g}$, the corresponding decomposition $\mathfrak{g}=\mathfrak{h}+\mathfrak{m}$ will be called a {\it bi-invariant
decomposition} of $\mathfrak{g}$ for the homogeneous space $G/H$.

The restriction to $T_o(G/H)$ defines a canonical one-to-one correspondence between $G$-homogeneous Finsler metrics on $G/H$ and
$\mathrm{Ad}(H)$-invariant Minkowski norm on $\mathfrak{m}$. For simplicity we will use the same $F$ to denote the corresponding Minkowski norm  on $\mathfrak{m}$.

The method of submersion can be applied to define homogeneous Finsler metrics on $G/H$.
Let $\bar{F}$ be a left invariant Finsler metric on $G$ which is right invariant under $H$. Then the following lemma indicates that there exists a uniquely defined subduced homogeneous metric
on $G/H$.
\begin{lemma} \label{lemma-3-3}
Keep all the above notations. Then there is a uniquely defined homogeneous metric $F$ on $G/H$ such that the tangent map
\begin{equation}
\pi_*:(T_g G,\bar{F}(g,\cdot))\rightarrow (T_{\pi(g)}(G/H),F(\pi(g),\cdot))
 \end{equation}
 is a submersion for any $g\in G$.
\end{lemma}

\begin{proof}
Let $\mathfrak{g}=\mathfrak{h}+\mathfrak{m}$ be an $\mathrm{Ad}(H)$-invariant
decomposition for $G/H$.
From the previous subsection, we see that the tangent map $\pi_*:T_e G\rightarrow T_o (G/H)$ defines a unique subduced Minkowski norm
$F$ on $\mathfrak{m}=T_o M$ from $\bar{F}(e,\cdot)$. Since $\bar{F}(e,\cdot)$ is
$\mathrm{Ad}(H)$-invariant, $F$ is also $\mathrm{Ad}(H)$-invariant. Then left translations
by $G$ defines a $G$-homogeneous metric on $M$, which is also denoted as $F$ for simplicity. Since $\pi_*|_{T_eG}$ is a submersion, $\bar{F}$ is left $G$-invariant, and $F$ is $G$-homogeneous. Then the map $\pi_*|_{T_g G}=g_*\circ\pi_*|_{T_eG}\circ(L_{g^{-1}})_*$
is a submersion between the Minkowski spaces $(T_g G,\bar{F}(g,\cdot))$ and $(T_{\pi(g)}M,F(\pi(g),\cdot))$, for any $g\in G$.
\end{proof}

On the other hand, the next lemma indicates that any homogeneous Finsler metric can be
subduced from a well-chosen Finsler metric $\bar{F}$ on a Lie group and the natural projection.

\begin{lemma}\label{lemma-3-4}
Let $F$ be a homogeneous metric on $G/H$,  $\mathfrak{g}=\mathfrak{h}+\mathfrak{m}$ an $\mathrm{Ad}(H)$-invariant
decomposition for $G/H$. Then there exists a
left $G$-invariant and right $H$-invariant metric $\bar{F}$ on $G$ such that
$\bar{F}|_\mathfrak{m}=F$ when they are viewed as Minkowski norms on $\mathfrak{g}$ and $\mathfrak{m}$ respectively, and
$F$ is subduced from $\bar{F}$ and the projection $\pi:G\rightarrow G/H$ when they are viewed as Finsler metrics.
\end{lemma}

\begin{proof}
We need to construct an $\mathrm{Ad}(H)$-invariant Minkowski norm
$\bar{F}$ satisfying the conditions of Lemma \ref{lemma-3-3}. The conditions that $\bar{F}|_\mathfrak{m}=F$
and that $F$ is subduced from $\bar{F}$ by the projection can be equivalently
stated as that the indicatrix $I_{\bar{F}}$ of $\bar{F}$ is tangent to the cylinder
$\mathfrak{h}\times I_{F}$ at each point of the indicatrix $I_F\subset\mathfrak{m}$
of $F$. A Minkowski norm $\tilde{F}$ on $\mathfrak{g}$ which satisfies this statement
but may not be $\mathrm{Ad}(H)$-invariant can be constructed inductively by the
following observation:

Any Minkowski norm $F$ on $\mathbb{R}^{n-1}=\{(y^1,\ldots,y^{n-1},0)|\forall y^1,\ldots,y^{n-1}\}\subset\mathbb{R}^n$ can be extended to a Minkowski norm
$\tilde{F}$ on $\mathbb{R}^n$ such that the indicatrix $I_{\tilde{F}}$ is tangent
to the cylinder $I_F\times \mathbb{R}$ at each point of the indicatrix $I_F$ of $F$.

To prove the above  assertion, we need a non-negative function
$\phi\in C[0,1]\cap C^\infty[0,1)$,
such that on $[1/2,1]$, $\phi(t)=\sqrt{1-t}$, and all derivatives $\phi^{(k)}$ for $k>0$ vanish at 0,
with $\phi'(t)<0$ and $\phi''(t)<0$ for $0<t<1$. We also need a smooth function $\psi$
on $\mathbb{R}^{n-1}$
such that $\psi$ is compactly supported in the closed ball $B^F_{1/2}(0)\subset\mathbb{R}^{n-1}$ with  $F$-radius $1/2$ and  center $0$, such that $0$ is a critical point of $\psi$ with a negative definite Hessian matrix there. When a positive number $\lambda$ is
sufficiently close to $0$, the set
\begin{equation}\label{0000}
\{y=(y',y^n)| y^n=\pm (\phi(F^2(y'))+\lambda\psi(y'))\}
\end{equation}
is a smooth hypersurface surrounding 0. Obviously it is tangent to
$I_F\times\mathbb{R}$ at each point of $I_F$. To see that it defines a Minkowski norm
$\tilde{F}$, we
need to check the convexity condition. Notice that when
$|y^n|<\frac{\sqrt{2}}{2}\phi(\frac12)F(y')$, it defines the function
$\sqrt{F^2(y')+(y^n)^2}$ with $y=(y',y^n)$, for which the convexity condition of the Minkowski norm is satisfied at its smooth points. For other points, the hypersurface (\ref{0000}) coincides
with the graphs of $\pm f$, in which $f(y')=\phi(F^2(y'))+\lambda\phi(y')$ for
$y'\in B^F_{3/4}(0)$. The Hessian of $\phi(F^2)$
is negative definite everywhere in $B^F_{3/4}(0)$ except that is is euqal to zero at $y'=0$, while the
Hessian of $\psi$ near $0$ is negative definite, hence  the Hessian of $f$ at each point of $B^F_{3/4}(0)$ is negative definite   when the positive $\lambda$ is sufficiently close to $0$. This argument proves the convexity condition at all points.

With the Minkowski norm $\tilde{F}$ on $\mathfrak{g}$ inductively constructed above, the average
\begin{equation}
\bar{F}(y)=\sqrt{\frac{\int_H \tilde{F}^2(\mathrm{Ad}(h)y)d\mathrm{vol}_H}{\int_H d\mathrm{vol}_H}}
\end{equation}
with respect to a bi-invariant volume form $d\mathrm{vol}_H$ defines an
$\mathrm{Ad}(H)$-invariant Minkowski norm. Since  $F$ is  $\mathrm{Ad}(H)$-invariant, the cylinder $\mathfrak{h}\times I_F$ is preserved under the  $\mathrm{Ad}(H)$-action. Therefore  $I_{\bar{F}}$ is tangent to $\mathfrak{h}\times I_F$ at each point of $I_F=0\times I_F\subset\mathfrak{g}$, i.e., it is
the Minkowski norm as required in  the lemma.
\end{proof}

\section{A flag curvature formula for homogeneous Finsler spaces}

The purpose of this section is to prove the following theorem.

\begin{theorem} \label{flag-curvature-formula}
Let $(G/H,F)$ be a connected homogeneous Finsler space, and $\mathfrak{g}=\mathfrak{h}+\mathfrak{m}$ be an $\mathrm{Ad}(H)$-invariant
decomposition for $G/H$. Then for any linearly independent commuting pair
$u$ and $v$ in $\mathfrak{m}$ satisfying
$
\langle[u,\mathfrak{m}],u\rangle^F_u=0
$,
we have
\begin{equation*}
K^F(o,u,u\wedge v)=\frac{\langle U(u,v),U(u,v)\rangle_u^F}
{\langle u,u\rangle_u^F \langle v,v\rangle_u^F-
{\langle u,v\rangle_u^F}\langle u,v\rangle_u^F},
\end{equation*}
 where $U(u,v)$ is the bi-linear symmetric mapping from $\mathfrak{m}\times \mathfrak{m}$ to $\mathfrak{m}$  determined by
\begin{equation*}
\langle U(u,v),w\rangle_u^F=\frac{1}{2}(\langle[w,u]_\mathfrak{m},v\rangle_u^F
+\langle[w,v]_\mathfrak{m},u\rangle_u^F), \mbox{ for any }w\in\mathfrak{m},
\end{equation*}
where $[\cdot,\cdot]_\mathfrak{m}=\mathrm{pr}_\mathfrak{m}\circ[\cdot,\cdot]$ and $\mathrm{pr}_\mathfrak{m}$ is the projection
with respect to the given $\mathrm{Ad}(H)$-invariant decomposition.
\end{theorem}

\subsection{A refinement of Theorem \ref{theorem-2-1}}
Theorem \ref{flag-curvature-formula} gives us an explicit formula to study the flag curvature of homogeneous Finsler spaces. Nevertheless, to get our classification, we need more elegant results. Theorem \ref{theorem-2-1} of Z.~Shen
provides a very enlightening observation. However, it is not convenient  in the homogeneous case, since a  Killing vector field on a homogeneous Finsler space
is generally not a geodesic field. On the other hand,
we only need to consider the geometric properties of a homogeneous Finsler space at the origin $o=\pi(e)$, and Killing vector fields whose integration curves at $o$ are geodesics can be easily found. This leads  to  the following refinement of Theorem \ref{theorem-2-1}.

\begin{theorem}\label{refine-shen-thm}
Let $Y$ be a vector field on a Finsler space $(M,F)$, such that $y=Y(p)\neq 0$, and $Y$
generates a geodesic of constant speed through $p$, then  $R^F_y=R^{g_Y}_y$. Furthermore, for any tangent plane
$\mathbf{P}$ in $T_p M$ containing $y$, we  have $K^F(p,y,\mathbf{P})=
K^{g_Y}(p,\mathbf{P})$.
\end{theorem}

\begin{proof}
Let $x=(x^i)$ and $y=y^j\partial_{x^j}$ be  a standard local coordinate system  defined on an open neighborhood $\mathcal{U}$ of  $p$,
such that $Y=\partial_{x^1}$. In the following, quantities with respect to
$g_Y=(\tilde{g}_{ij}(\cdot))=(g_{ij}(Y(\cdot)))$ will be denoted with a tilde.

The  covariant derivatives with respect to $F$ can be expressed as
\begin{eqnarray}
\nabla_Y^Y Y &=& 2G^i(Y)\partial_{x^i}\nonumber\\
&=& \frac{1}{2}[g^{il}(2\partial_{x^k}g_{lj}-
\partial_{x^l}g_{jk})](Y)y^jy^k|_{y^1=1,y^2=\cdots=y^n=0}\partial_{x^i}.
\end{eqnarray}
Similarly,  the  covariant derivatives with respect to $g_Y$ can be expressed as
\begin{eqnarray}
\tilde{\nabla}_Y^Y Y &=& 2\tilde{G}^i(Y)\partial_{x^i}\nonumber\\
&=& \frac{1}{2}\tilde{g}^{il}(2\partial_{x^k}\tilde{g}_{lj}-\partial_{x^l}\tilde{g}_{jk})
y^jy^k|_{y^1=1,y^2=\cdots=y^n=0}\partial_{x^i},
\end{eqnarray}
where $n=\dim M$. Now, for $Y=\partial_{x^1}$, we have
\begin{equation}
(\partial_{x^k}g_{lj})(Y)=\partial_{x^k}\tilde{g}_{lj},
(\partial_{x^l}g_{jk})(Y)=\partial_{x^l}\tilde{g}_{jk},
\mbox{ and }g^{il}(Y)=\tilde{g}^{il}.
\end{equation}
Then the equalities $G^i(Y)=\tilde{G}^i(Y)$ hold on $\mathcal{U}$, $\forall i$. Thus on $\mathcal{U}$, we have
\begin{equation}
[\partial_{x^k}G^i](Y)=\partial_{x^k}[G^i(Y)]=
\partial_{x^k}[\tilde{G}^i(Y)]=[\partial_{x^k}\tilde{G}^i](Y),
\end{equation}
for any $i$ and $k$. Since the integration curve of $Y$  at $p$ is a geodesic, we have $G^i(Y)=\tilde{G}^i(Y)=0$, $\forall i$, on this integration curve.

Now we show that at the point $p$, for any $i,j$, the quantityies $N_j^i(Y)=[\partial_{y^j}G^i](Y)$ and $\tilde{N}_j^i(Y)=[\partial_{y^j}\tilde{G}^i](Y)$ are equal. In fact,
\begin{equation}
N_j^i=\frac{1}{2}g^{il}[\partial_{x^k}g_{jl}+\partial_{x^j}g_{kl}-
\partial_{x^l}g_{jk}]y^k-2g^{il}C_{jkl}G^k.
\end{equation}
Thus  at the point $p$ we have
\begin{eqnarray}
N_j^i(Y) &=& \frac{1}{2}[g^{il}(\partial_{x^1}g_{jl}+
\partial_{x^j}g_{1l}-\partial_{x^l}g_{j1})](Y)\nonumber\\
&=&\frac{1}{2}[\tilde{g}^{il}(\partial_{x^1}\tilde{g}_{jl}+
\partial_{x^j}\tilde{g}_{1l}-\partial_{x^l}\tilde{g}_{j1})](Y)\nonumber\\
&=& \tilde{N}_j^i(Y),
\end{eqnarray}
which proves our assertion.
A similar calculation then shows that in $\mathcal{U}$, we also have
\begin{equation}
N_j^i(Y)=\tilde{N}_j^i(Y)-2g^{il}(Y)C_{ljk}(Y)G^k(Y).
\end{equation}
 Since for any $i$,  $G^i(p, Y)=0$, the equalities
\begin{equation}
[\partial_{x^1}N_j^i](Y)=[\partial_{x^1}\tilde{N}_j^i](Y)-2g^{il}(Y)C_{ljk}(Y)
[\partial_{x^1}]G^k(Y)
\end{equation}
hold at $p$.
Now our previous argument shows that, for any $k$,  $$[\partial_{x^1}G^k](p,Y)=\partial_{x^1}G^k(p,Y)=0.$$
Therefore at $p$, we have $\partial_{x^1}N_j^i(Y)=\partial_{x^1}\tilde{N}_j^i(Y)$.

 Comparing the formula (\ref{formula-2-6}) for the Riemannian curvature
$R_k^i(y)$ and $\tilde{R}_k^i(y)$ with $y=Y(p)$, we  get $R^F_y=R^{g_Y}_y$.
Then the equality for the flag curvatures follows immediately.
\end{proof}
\subsection{The Finslerian submersion technique}
Let $(G/H,F)$ be a homogeneous
Finsler space, and $\mathfrak{g}=\mathfrak{h}+\mathfrak{m}$ an $\mathrm{Ad}(H)$-invariant decomposition for $G/H$. Let $\bar{F}$ be the Finsler metric on $G$ as in Lemma
\ref{lemma-3-4}, i.e., $\bar{F}$ is left $G$-invariant and right $H$-invariant,
  $\bar{F}|_\mathfrak{m}=F$ when $\bar{F}$ and $F$ are viewed as Minkowski norms on
$\mathfrak{g}$ and $\mathfrak{m}$ respectively, and  $F$ is subduced from $\bar{F}$ and the
natural projection $\pi$. We keep all the other notations of the previous section.

Let $u$ and $v$ be two linearly independent tangent vectors in $\mathfrak{m}=T_o(G/H)$ and assume
that $u$ satisfies the condition
\begin{equation}\label{formula-4-22}
\langle [u,\mathfrak{m}]_\mathfrak{m},u\rangle^F_{u}=0.
\end{equation}
In the following we present a method to calculate the flag curvature $K^F(o,u,u\wedge v)$.

First note that there exist a left invariant vector field $U_1$ and a right invariant vector field $U_2$ on $G$ with
$U_1(e)=U_2(e)=u\in\mathfrak{m}\subset\mathfrak{g}=T_e G$. Note also that $U_2$ can be
pushed down through the projection map which defines a Killing vector field
$U$ of $(G/H,F)$. Let $U'$ be the horizonal lift of $U$ with respect to the submersion
$\pi:(G,\bar{F})\rightarrow(G/H,F)$. Since $U(o)=u\neq 0$, $U'$ is smooth on an open neighborhood of $e$. Further, there exist an open neighborhood $\mathcal{U}'$ of
$e$ and $\mathcal{U}$ of $o$, such that $\pi$ is a Riemannian submersion between
$(\mathcal{U}',g_{U'})$ and $(\mathcal{U},g_U)$.

Let $V_1$ be the left invariant vector field on $G$ extending
$v\in\mathfrak{m}\subset\mathfrak{g}=T_e G$. Let $V'$ be the horizonal shift of $V_1$ in $\mathcal{U}'$
 with respect to the submersion $\pi:(\mathcal{U}',g_{U'})\rightarrow(\mathcal{U},g_U)$. When $\mathcal{U}'$ and $\mathcal{U}$ are sufficiently small, $V'$ is smooth and non-vanishing
 in $\mathcal{U}'$.

We first deduce the following lemma.

\begin{lemma}\label{lemma-4-2}
The following assertions holds:
\begin{description}
\item{\rm (1)} The horizonal lift of $u$ with respect to the submersion $\pi_*:(\mathfrak{g},\bar{F})
\rightarrow(\mathfrak{m},F)$ is exactly $u$ itself.
\item{\rm (2)} The vector field $U$ generates a geodesic of $(G/H,F)$ through $o$. Moreover,  the
vector fields $U'$, $U_1$ and $U_2$ generate the same geodesic of $(G,\bar{F})$ through $e$.
\item{\rm (3)} The vectors $u$ and $v$ in $\mathfrak{m}$ are the horizonal lift of themselves with respect to the
submersion $\pi_*:(\mathfrak{g},\langle\cdot,\cdot\rangle^{\bar{F}}_u)
\rightarrow(\mathfrak{m},\langle\cdot,\cdot\rangle^F_u)$.
\end{description}
\end{lemma}

\begin{proof}
(1)\quad By Lemma \ref{lemma-3-4},  $F$ is both the restriction and the
subduced metric of $\bar{F}$. Then the assertion in (1) follows.

(2)\quad Without losing generality, we can  assume that $u$ is a unit vector for $\bar{F}$. Then the assumption (\ref{formula-4-22}) and the property of $\bar{F}$ implies that the $\mathrm{Ad}(G)$-orbit of $u$
is tangent to both the indicatrix $I_{\bar{F}}$ of $\bar{F}$
and
the cylinder $\mathfrak{h}\times I_{F}$ for the indicatrix $I_F$ of $F$ at $u$.
The ${F}$-length of $U$ at any $\pi(g)$ is ${F}(\mathrm{pr}_{\mathfrak{m}}(\mathrm{Ad}(g)u))$. So for $g=\exp (tX)$, $X\in\mathfrak{g}$, the smooth function $F(\mathrm{pr}_{\mathfrak{m}}(\mathrm{Ad}(\exp (tX))u))$ of $t$ has  zero
derivative at $t=0$, that is,
\begin{equation}
F(\mathrm{pr}_{\mathfrak{m}}(\mathrm{Ad}(\exp (tX))u))=F(u)+o(t),
\end{equation}
 where $o(t)$ denotes an infinitesimal quantity of higher order then $t$.
This implies that the $F$-length function
of $U$ has a critical point at $o$. Since $U$ is a Killing vector field,
 by Lemma 3.1 of \cite{DengXu2012},  $U$ generates a geodesic of $(G/H,{F})$ through the critical point $o$.
Similarly,   the vector field $U_2$ is a Killing vector field for $(G,\bar{F})$, and its $\bar{F}$-length function has a critical point at $e$, so it generates a geodesic
of $(G,\bar{F})$ through $e$. It is easily seen that $U_1$ and $U_2$ generate the  same geodesic $\exp(tu)$.
Since $U'$
is the smooth horizonal lift of $U$ on an open neighborhood of $e$,  $U'$ also generates a
geodesic of $(G,\bar{F})$ at $e$. By (1) of this lemma, $U_1(e)=U'(e)=u$,
 so $U'$ also generates the  geodesic $\exp(tu)$ through $e$.

(3)\quad At any point in $\mathfrak{m}$, the derivative of $\bar{F}^2$ is equal to $0$ in the directions
of $\mathfrak{h}$. So by (\ref{formula-2-2}), for any $ w\in\mathfrak{h}$, we have
\begin{equation}
\langle v,w\rangle^{\bar{F}}_u=\frac{1}{2}
\frac{\partial}{\partial s}[\frac{\partial}{\partial t}
\bar{F}^2(u+sv+tw)|_{t=0}]|_{s=0}=0,
\end{equation}
which proves the assertion for $v$. The assertion for $u$ is obvious.
\end{proof}

By (3) of Lemma \ref{lemma-4-2},
the sectional curvature formula
for the Riemannian submersion
$
\pi:(\mathcal{U}',g_{U'})\rightarrow (\mathcal{U},g_U)
$ indicates that
\begin{equation}\label{formula-4-26}
K^{g_U}(o,u,u\wedge v)=K^{g_{U'}}(e,u,u\wedge v)+\frac{3||A(u,v)||^2}{||u\wedge v||^2},
\end{equation}
in which the norms of vectors are defined by the inner product $\langle\cdot,\cdot\rangle_u^{\bar{F}}$ from $g_{U'}(e,\cdot)=g_{U_1}(e,\cdot)$, and $A(u,v)$ is the value of
the vertical component of $\frac{1}{2}[U',V'](e)$.  Now applying (2) of Lemma
\ref{lemma-4-2} and Theorem \ref{refine-shen-thm}, we have
$K^{g_U}(o,u,u\wedge v)=K^F(o,u,u\wedge v)$ and
$K^{g_{U'}}(e,u,u\wedge v)=K^{\bar{F}}(e,u,u\wedge v)=K^{g_{U_1}}(e,u,u\wedge v)$.
This proves the following
\begin{lemma}\label{lemma-4-3}
Keep all the notations as above. Then
\begin{equation}\label{3000}
K^{F}(o,u,u\wedge v)=K^{g_{U_1}}(e,u\wedge v)+\frac{3||A(u,v)||^2}{||u\wedge v||^2}.
\end{equation}
\end{lemma}

Since both $\bar{F}$ and $U_1$ are left $G$-invariant, $g_{U_1}$ is a left
invariant Riemannian metric on $G$, defined by the inner product $\langle\cdot,\cdot\rangle_u^{\bar{F}}$ on $\mathfrak{m}$, hence $K^{g_{U_1}}(e,u\wedge v)$ can be
expressed with the sectional curvature formula of Riemannian homogeneous spaces, which is very simple and useful in the case that $u$ and $v$ are a commuting
pair in $\mathfrak{m}$.

\subsection{The proof of Theorem \ref{flag-curvature-formula}}
Now we continue the argument of the previous subsection.
 Assume further that $[u,v]=0$. Then we have
\begin{equation}\label{4000}
K^{g_{U_1}}(e,u\wedge v)=\frac{\langle U(u,v),U(u,v)\rangle_u^{\bar{F}}
-\langle U(u,u),U(v,v)\rangle_u^{\bar{F}}}{\langle u,u\rangle_u^{\bar{F}}
\langle v,v\rangle_u^{\bar{F}}-{\langle u,v\rangle_u^{\bar{F}}}
\langle u,v\rangle_u^{\bar{F}}},
\end{equation}
where $U:\mathfrak{g}\times\mathfrak{g}\rightarrow\mathfrak{g}$ is the bi-linear symmetric mapping defined by
\begin{equation*}
\langle U(w_1,w_2),w_3\rangle_u^{\bar{F}}=\frac{1}{2}
(\langle[w_3,w_1],w_2\rangle_u^{\bar{F}}+\langle[w_3,w_2],w_1
\rangle_u^{\bar{F}}).
\end{equation*}
Now we prove that $U(u,v)\in\mathfrak{m}$. By
Theorem 1.3 of \cite{DengHou2004}, we have
\begin{equation}\label{5000}
	\langle[w,u],v\rangle^{\bar{F}}_u+\langle u,[w,v]\rangle^{\bar{F}}_u
	+2C_u^{\bar{F}}(u,v,[w,u])=0,
\end{equation}
for any $w\in\mathfrak{h}$. Thus the Cartan tensor term $C_u^{\bar{F}}(u,\cdot,\cdot)$ vanishes. So $\langle U(u,v),\mathfrak{h}\rangle_u^{\bar{F}}=0$, i.e., $U(u,v)\in\mathfrak{m}$,
 and it can be determined by
\begin{eqnarray*}
\langle U(u,v),w\rangle_u^F &=& \langle U(u,v),w\rangle_u^{\bar{F}}\\
&=&\frac12(\langle [w,u],v\rangle_u^{\bar{F}}
+\langle [w,v],u\rangle_u^{\bar{F}})\\
&=&\frac12(\langle [w,u]_\mathfrak{m},v\rangle_u^{{F}}
+\langle [w,v]_\mathfrak{m},u\rangle_u^{{F}}),
\end{eqnarray*}
for any $w\in\mathfrak{m}$.

Using a similar equality as (\ref{5000}), we can prove that
$\langle [u,\mathfrak{h}],u\rangle_u^{\bar{F}}=0$. So we have
\begin{eqnarray*}
\langle [u,\mathfrak{g}],u\rangle_u^{\bar{F}} &\subset&
\langle [u,\mathfrak{m}],u\rangle_u^{\bar{F}}+\langle [u,\mathfrak{h}],u\rangle_u^{\bar{F}}\\
&\subset&
\langle [u,\mathfrak{m}]_\mathfrak{m},u\rangle_u^{\bar{F}}
=\langle[u,\mathfrak{m}]_\mathfrak{m},u\rangle_u^F\\
&=&0.
\end{eqnarray*}
Thus $U(u,u)=0$.

To summarize, (\ref{4000}) can be simplified as
\begin{equation*}
K^{g_{U_1}}(e,u\wedge v)=\frac{\langle U(u,v),U(u,v)\rangle_u^{F}}{\langle u,u\rangle_u^{F}
\langle v,v\rangle_u^{F}-{\langle u,v\rangle_u^{F}}\langle u,v\rangle_u^{F}},
\end{equation*}
where $U(u,v)\in\mathfrak{m}$ is the bi-linear symmetric mapping defined by
\begin{equation}\label{5999}
\langle U(u,v),w\rangle_u^{F}=\frac{1}{2}
(\langle[w,u]_\mathfrak{m},v\rangle_u^{F}+\langle[w,v]_\mathfrak{m},u
\rangle_u^{F}).
\end{equation}

To prove Theorem \ref{flag-curvature-formula}, we only need to prove
that the $A(u,v)$-term  in Lemma \ref{lemma-4-3} vanishes when $[u,v]=0$.
Let $X_1,\ldots,X_N$ be a basis of the space of all left invariant
vector fields on $G$ such that $X_1(e)=u$ and $X_2(e)=v$ and denote
the smooth vector fields $U'$ and $V'$ around $e$ as
$U'(g)=u^i(g)X_i$ and $V'(g)=v^i(g)X_i$. Then we have
\begin{equation}
[U',V'](e)=[u,v]+\frac{d}{dt}v^i(\exp(tu))|_{t=0}X_i-\frac{d}{dt}u^i(\exp(tv))|_{t=0}X_i.
\end{equation}
Since $[u,v]=0$, $U'(\exp(tv))\equiv X_1$ and $V'(\exp(tu))\equiv X_2$, we have
$[U',V'](e)=0$. Thus its vertical factor $A(u,v)$ vanishes as well, which completes
the proof of Theorem \ref{flag-curvature-formula}.

\subsection{An intrinsic proof of Theorem \ref{flag-curvature-formula}}
The metric $\bar{F}$ we have constructed on $G$ is for building up the Finslerian submersion technique. But it is not intrinsic for the geometry of
the homogeneous space, and it does not appear in the flag curvature formula. This fact implies that there should be an intrinsic
proof of Theorem \ref{flag-curvature-formula}.

In \cite{Huang2013},  the third author of this work
uses invariant frames to give explicit formulas for curvatures of homogeneous Finsler spaces.
This method can be used to give an intrinsic proof of the theorem.

Recall that  the {\it spray vector field} $\eta:\mathfrak{m}\backslash\{0\}\rightarrow\mathfrak{m}$
 is defined by
\begin{equation*}
	\langle\eta(u), w\rangle_u^F=\langle u,[w,u]_{\mathfrak{m}}\rangle_u^F, \quad \forall v\in\mathfrak{m}.
\end{equation*}
Meanwhile, the {\it connection operator} $N:(\mathfrak{m}\backslash\{0\})\times\mathfrak{m}\rightarrow\mathfrak{m}$
is a linear operator
on $\mathfrak{m}$ determined by
\begin{equation*}
	\begin{aligned}
	2\langle N(u,w_1),w_2\rangle_u^F=\langle [w_2,w_1]_{\mathfrak{m}}, u\rangle_u^F & + \langle [w_2,u]_{\mathfrak{m}}, w_1\rangle_u^F+\langle [w_1,u]_{\mathfrak{m}},w_2\rangle_u^F\\
	& - 2C^F_u(w_1,w_2,\eta(u)),\quad \forall w_1,w_2\in\mathfrak{m}.
	\end{aligned}	
\end{equation*}
Using these two notions, Huang proved the following formula
for Riemannian curvature $R_u: T_o(G/H)\rightarrow T_o(G/H)$,
\begin{equation}\label{6000}
	\langle R_u(w),w\rangle_u^F
	=\langle [[w,u]_{\mathfrak{h}},w],u\rangle_u^F
	+\langle \tilde{R}(u)w,w\rangle_u^F,\quad w\in\mathfrak{m},
\end{equation}
where the linear operator $\tilde{R}(u):T_o(G/H)\rightarrow T_o(G/H)$
is given by
$$\tilde{R}(u)=D_{\eta(u)}N(u,w)-N(u,N(u,w))+
N(u,[u,w]_\mathfrak{m})-[u,N(u,w)]_\mathfrak{m},$$
 and $D_{\eta(u)}N(u,w)$ is the derivative of $N(\cdot,w)$ at $u\in\mathfrak{m}\backslash\{0\}$ in the direction of $\eta(u)$ (in particular,
it is $0$ when $\eta(u)=0$).

Now suppose that $u\in\mathfrak{m}\backslash\{0\}$ satisfies (\ref{formula-4-22}), i.e., $\langle[u,\mathfrak{m}],u\rangle_u^F=0$. Then it is easy to see that
$\eta(u)=0$. For any $v\in\mathfrak{m}$ which commutes with $u$,
we have $N(u,v)=U(u,v)$ in (\ref{5999}). Thus
\begin{eqnarray*}
\langle R_u (v),v\rangle_u^F &=&
-\langle N(u,N(u,v)),v\rangle_u^F
-\langle [u,N(u,v)],v\rangle_u^F\\
&=& -\frac12(\langle [v,N(u,v)]_\mathfrak{m},u\rangle_u^F+
\langle[N(u,v),u]_\mathfrak{m},v\rangle_u^F)
+\langle [N(u,v),u],v\rangle_u^F\\
&=& \frac12(\langle[N(u,v),v],u\rangle_u^F+
\langle[N(u,v),u],v\rangle_u^F)\\
&=&\langle U(u,v),N(u,v)\rangle_u^F=\langle U(u,v),U(u,v)\rangle_u^F.
\end{eqnarray*}
From this calculation, the flag curvature formula in Theorem \ref{flag-curvature-formula} follows immediately.

\section{Proof of Theorem \ref{thm1}}
In this section we complete the proof of Theorem \ref{thm1}.
Let $F$ be a positively curved homogeneous metric on the compact homogeneous space $G/H$, in which $G$ is a compact
connected simply connected Lie group, and $H$ is a closed connected subgroup of $G$.
Keep all relevant notations. Let $\mathfrak{g}=\mathfrak{h}+\mathfrak{m}$
be the orthogonal decomposition with respect to a bi-invariant inner product
on $\mathfrak{g}$.
Let $\mathfrak{t}$ be a Cartan subalgebra of $\mathfrak{g}$ such that
$\mathfrak{t}\cap\mathfrak{h}$ is a Cartan subalgebra of $\mathfrak{h}$.
Let $T$ and $T_H$ be the maximal torus subgroup of $G$ and $H$ corresponding to $\mathfrak{t}$ and $\mathfrak{t}\cap\mathfrak{h}$ respectively.

We will use the following lemma to prove the inequality between the ranks of $G$ and $H$ (see Theorem \ref{theorem-5-2}).

\begin{lemma}\label{lemma-5-1}
Keep all the above notations.  Let $N_G(T_H)$ and $N_H(T_H)$ be the normalizers of $T_H$ in $G$ and $H$ respectively, i.e.,
\begin{eqnarray*}
N_G(T_H)&=&\{g\in G|g T_H g^{-1}=T_H\}, \mbox{ and}\\
N_H(T_H)&=&\{g\in H|g T_H g^{-1}=T_H\}=N_G(T_H)\cap H.
\end{eqnarray*}
Then the $\mathrm{Ad}(N_G(T_H))$-orbit $N_G(T_H)\cdot o=N_G(T_H)/N_H(T_H)\subset G/H$
of $o=\pi(e)$ consists of all the common fixed
points of $T_H$.
\end{lemma}
\begin{proof}
Given $x\in N_G(T_H)\cdot o$, let $x=g\cdot o$ for some
$g\in N_G(T_H)$.  Then by the definition of normalizer,
for any $g_1\in T_H$, there exists $g_2\in T_H\subset H$ such that
$g_1 g=g g_2$. Hence we have
$$g_1\cdot x=g_1 g\cdot o=gg_2\cdot o=g\cdot o=x.$$
This shows that $x$ is fixed by all the elements in $T_H$.

Conversely, suppose $x\in G/H$ is fixed by any element in $T_H$. We shall
prove that $x\in N_G(T_H)\cdot o$.  Suppose $x=g\cdot o$ for some $g\in G$,
such that $g_1\cdot x=x$ for all $g_1\in T_H$. Then
$g^{-1} g_1 g \in H$, $\forall g_1\in T_H$, i.e., $g^{-1} T_H g\subset H$.
There exists $g_2\in H$, such that
$(gg_2)^{-1}T_H (gg_2)=g_2^{-1}(g^{-1}T_H g)g_2=T_H$, i.e.
$gg_2\in N_G(T_H)$. Thus $x= g\cdot o=gg_2\cdot o\in N_G(T_H)\cdot o$.
This completes the proof of the lemma.
\end{proof}

The following rank equality is a crucial observation for the classification of
positively curved homogeneous spaces in both Riemannian geometry and Finsler
geometry.

\begin{theorem}\label{theorem-5-2}
Let $G$ be a compact connected simply connected Lie group and $H$ a closed
subgroup such that $G$ acts effectively on $G/H$.  Assume $G/H$ admits a
$G$-homogeneous Finsler metric $F$ with positive flag curvature. Then $\mathrm{rk}G\leq \mathrm{rk}H+1$.
Moreover, we have the following.
\begin{enumerate}
\item The orbit of $N_G(T_H)$ is totally geodesic in $(G/H,F)$.
\item If $\dim G/H$ is even, then $T=T_H$, i.e.,
$\mathrm{rk}G=\mathrm{rk}H$.
\item If $\dim G/H$ is odd, then the identity component of $N_G(T_H)/T_H$ is isomorphic to
$\mathrm{U}(1)$, $\mathrm{SU}(2)$ or $\mathrm{SO}(3)$.  In this case
$\mathrm{rk}G=\mathrm{rk}H+1$.
\end{enumerate}
\end{theorem}
\begin{proof}
By Lemma \ref{lemma-5-1} and the main theorem of \cite{Deng2008}, the compact smooth orbit $N_G(T_H)\cdot o$, i.e., the common fixed points of $T_H$,
is a (possibly disconnected) totally geodesic sub-manifold.  Thus the
homogeneous metric $F|_{N_G(T_H)\cdot o}$ on the orbit
$N_G(T_H)\cdot o$ is also positively curved whenever
the dimension is bigger than $1$.
In fact, the identity component of $N_H(T_H)$ is  $T_H$, so
$N_G(T_H)\cdot o=N_G(T_H)/N_H(T_H)$ is finitely covered by the compact group
$G'=N_G(T_H)/T_H$, and $F$ induces a positively curved left invariant Finsler
metric on $G'$ when $\dim G'>1$. By Proposition 5.3 of \cite{HuDeng2013},
$\dim G'>0$ if and only if $G/H$ is odd dimensional. If $\dim G'=1$, then the
connected component of $G'$ is $\mathrm{U}(1)$.
If $\dim G'>1$, then by Theorem 5.1
of \cite{HuDeng2013}, the connected component of $G'$ is $\mathrm{SU}(2)$
or $\mathrm{SO}(3)$. The Lie algebra of $G'$ is the centralizer of $\mathfrak{t}\cap\mathfrak{h}$ in $\mathfrak{m}$. This proves the fact
$\mathrm{rk}G\leq \mathrm{rk}H+1$, as well as all the other statement of the theorem.
\end{proof}

Now we  assume further that $G/H$ is even dimensional.
Then there is a Cartan subalgebra $\mathfrak{t}$ contained in
$\mathfrak{h}$. Suppose we have the following decomposition with respect to $\mathrm{Ad}(T)$-actions:
\begin{equation}
\mathfrak{g}=\mathfrak{t}+\sum_{\alpha\in\Delta^+}\mathfrak{g}_{\pm\alpha},
\end{equation}
where each $\mathfrak{g}_{\pm\alpha}$ is $2$-dimensional (called a root plane).  Since $\mathfrak{t}$ is contained in $\mathfrak{h}$,
a root plane $\mathfrak{g}_{\pm\alpha}$ is
 contained either in $\mathfrak{h}$ or in $\mathfrak{m}$.  For simplicity,
we will just call $\pm\alpha$ roots of $\mathfrak{h}$
or $\mathfrak{m}$ accordingly. Denote the set of all roots of $\mathfrak{h}$ (or $\mathfrak{m}$) as $\Delta_\mathfrak{h}$ (or
$\Delta_\mathfrak{m}$).

Let $\alpha$ be a root of $\mathfrak{m}$ and $u$ a nonzero vector
in $\mathfrak{g}_{\pm\alpha}$. We now show that $u$ satisfies the condition in Theorem \ref{flag-curvature-formula}, namely,  the following lemma holds:

\begin{lemma}\label{lemma-5-3}
For any root $\alpha$ of $\mathfrak{m}$, and nonzero vector $u$
in $\mathfrak{g}_{\pm\alpha}$, we have
$\langle u,\mathfrak{g}_{\pm\gamma} \rangle_u^F=0$ for any root $\gamma$
of $\mathfrak{m}$ with $\gamma\neq\pm\alpha$. In particular,
$\langle [u,\mathfrak{m}]_\mathfrak{m},u\rangle=0$.
\end{lemma}

\begin{proof} Let $\mathfrak{t}'=\ker \alpha$ and $T'$ the sub-torus of $T_H\subset H$ generated by $\mathfrak{t}'$. Since both $F$ and $u$ are
$\mathrm{Ad}(T')$-invariant, so does the inner product $\langle\cdot,\cdot\rangle_u^F$ on $\mathfrak{m}$. For any root $\gamma$ of
$\mathfrak{m}$ with $\gamma\neq \pm\alpha$, there exists $g\in T'$ such
that $\mathrm{Ad}(g)|_{\mathfrak{g}_{\pm\gamma}}=-\mathrm{id}$. Thus
for any $w\in \mathfrak{g}_{\pm\gamma}$, we have
\begin{eqnarray*}
\langle w,u\rangle_u^F=\langle\mathrm{Ad}(g)w,\mathrm{Ad}(g)u\rangle_u^F
=-\langle w,u\rangle_u^F,
\end{eqnarray*}
i.e., $\langle u,\mathfrak{g}_{\pm\gamma}\rangle_u^F=0$.
Since
$$[u,\mathfrak{m}]_\mathfrak{m}\subset\mathop\sum\limits_{\gamma\in\Delta_\mathfrak{m},
\gamma\neq\pm\alpha}
\mathfrak{g}_{\pm\gamma},$$
 we also have
$\langle [u,\mathfrak{m}],u\rangle_u^F=0$. This completes the proof of the lemma.
\end{proof}

The following lemma is important for finding the specific vector  $v$ in Theorem \ref{flag-curvature-formula}.
\begin{lemma}\label{lemma-5-4}
Let $\alpha$ and $\beta$ be two roots of $\mathfrak{m}$
such that $\alpha\neq \pm\beta$ and $\alpha\pm\beta$ are not roots of $\mathfrak{g}$. Then for any root $\gamma$ of $\mathfrak{m}$ with
$\gamma\neq\pm\beta$, we have
$\langle \mathfrak{g}_{\pm\beta},\mathfrak{g}_{\pm\gamma}\rangle_u^F=0$.
\end{lemma}

\begin{proof}
Let $\mathfrak{t}'$ and $T'$ be defined as in the proof of Lemma \ref{lemma-5-3}. Since the inner product $\langle\cdot,\cdot\rangle_u^F$
on $\mathfrak{m}$ is $\mathrm{Ad}(T')$-invariant, corresponding to different irreducible representations of $T'$, and with respect to this inner product, $\mathfrak{m}$ can be orthogonally decomposed as the sum
of $\hat{\mathfrak{m}}_0=\mathfrak{g}_{\pm\alpha}$ (for the trivial representation of $T'$) and
\begin{eqnarray*}
\hat{\mathfrak{m}}_{\pm\gamma'}=\sum_{\gamma|_{\mathfrak{t}'}=\gamma',
\gamma\in\Delta_\mathfrak{m}}
\mathfrak{g}_{\pm\gamma},
\end{eqnarray*}
for all
 different nonzero pairs $\{\pm\gamma'\}$ such that $\{\pm\gamma'\}\subset{\mathfrak{t}'}^*\backslash\{0\}$ defines an irreducible representation of $T'$.
For each nonzero $\hat{\mathfrak{m}}_{\pm\gamma'}$ with $\gamma'\neq 0$, we have
\begin{eqnarray*}
\hat{\mathfrak{m}}_{\pm\gamma'}&\subset&\hat{\mathfrak{g}}_{\pm\gamma'}
=\sum_{\gamma|_{\mathfrak{t}'}=\gamma'}\mathfrak{g}_{\pm\gamma}\\
&=&\mathfrak{g}_{\pm\tau}+
\mathfrak{g}_{\pm(\tau+\alpha)}+
\cdots+\mathfrak{g}_{\pm(\tau+m\alpha)},
\end{eqnarray*}
where $\tau$, $\tau+\alpha$, $\cdots,\tau+m\alpha$ are all roots of $\mathfrak{g}$.
Let $\beta'=\beta|_{\mathfrak{t}'}\in{\mathfrak{t}'}^*$. Then from the conditions for $\beta$ in
the lemma, it is easily seen that $\beta'\neq 0$, and that
$\hat{\mathfrak{m}}_{\pm\beta'}=\hat{\mathfrak{g}}_{\pm\beta'}
=\mathfrak{g}_{\pm\beta}$.
So for $\gamma\in\Delta_\mathfrak{m}$ and $\gamma\notin\{\pm\alpha,\pm\beta\}$, we have $\langle \mathfrak{g}_{\pm\beta},\mathfrak{g}_{\pm\gamma}\rangle_u^F=0$.
On the other hand, the assertion for $\gamma=\pm\alpha$ has been proven in Lemma \ref{lemma-5-3}.
\end{proof}

We now recall the definition of {\it Condition (A)} defined by  N. Wallach, which is the key to his classification (see Proposition 6.1 of \cite{Wallach1972}). Keeping the notations as above, we say that the pair $(G,H)$ (or $(\mathfrak{g}, \mathfrak{h})$) satisfies Condition (A),  if for any two roots  $\alpha$,
$\beta$ of $\mathfrak{m}$ with $\alpha\neq\pm\beta$,
either  $\alpha+\beta$ or $\alpha-\beta$ is a root of $\mathfrak{g}$. The following lemma is the key to prove
Theorem \ref{thm1}.

\begin{lemma}\label{(A)}
If the even dimensional homogeneous space $G/H$ admits a $G$-homogeneous Finsler metric
$F$ of positive flag curvature, then the pair $(G,H)$ satisfies
condition (A).
\end{lemma}

\begin{proof}
Suppose conversely that $(G,H)$ does not satisfy Condition (A), i.e., there are
roots $\alpha$ and $\beta$ of $\mathfrak{m}$ such that $\beta\neq\alpha$ and  $\alpha\pm\beta$ are not roots of $\mathfrak{g}$. Select a nonzero vector $u\in\mathfrak{g}_{\pm\alpha}$ and a nonzero vector $v\in\mathfrak{g}_{\pm\beta}$. By Lemma \ref{lemma-5-3}, $K^F(o,u,u\wedge v)$
can be given by the flag curvature formula in Theorem \ref{flag-curvature-formula}. We only need to show that
$U(u,v)=0$, where $U(u,v)$ is determined by 
$$\langle U(u,v),w\rangle_u^F=\frac12(\langle[w,v]_\mathfrak{m},u\rangle_u^F
+\langle[w,u]_\mathfrak{m},v\rangle_u^F),\,\,\mbox{ for all }w\in\mathfrak{m}.$$
For this we only need to prove that the right side of the above equality vanishes for any $w\in\mathfrak{m}$. By our assumptions for the roots $\alpha$ and $\beta$, both
$[\mathfrak{g}_{\pm\alpha},\mathfrak{g}_{\pm\gamma}]_\mathfrak{m}$
and $[\mathfrak{g}_{\pm\beta},\mathfrak{g}_{\pm\gamma}]_\mathfrak{m}$
are contained in the subspace 
$$\sum_{\tau\in\Delta_\mathfrak{m},\tau\notin\{\pm\alpha,\pm\beta\}}
\mathfrak{g}_{\pm\tau},$$ for any root $\gamma$ of $\mathfrak{m}$.
So by Lemma \ref{lemma-5-3}, $\langle[\mathfrak{m},v]_\mathfrak{m},u\rangle_u^F=0$,
and by Lemma \ref{lemma-5-4}, $\langle[\mathfrak{m},u]_\mathfrak{m},v\rangle_u^F=0$.
Thus $U(u,v)=0$ and $K^F(o,u,u\wedge v)=0$, which is a contradiction to the
positive curvature assumption for $G/H$.
\end{proof}

Condition (A) is a strong restriction for the pair $(\mathfrak{g},\mathfrak{h})$ of compact Lie algebra and its subalgebra with maximal rank. For example, if we have a decomposition of $\mathfrak{g}$ as
the direct sum of ideals,
$$\mathfrak{g}=\mathfrak{g}_1\oplus\cdots\oplus\mathfrak{g}_n\oplus\mathbb{R}^m,$$
in which each $\mathfrak{g}_i$ is simple. Then we have also the decomposition for $\mathfrak{h}$ accordingly,
$$\mathfrak{h}=\mathfrak{h}_1\oplus\cdots\oplus\mathfrak{h}_n\oplus\mathbb{R}^m,$$
in which each $\mathfrak{h}_i$ is a subalgebra of $\mathfrak{g}_i$ with
$\mathrm{rk}\mathfrak{h}_i=\mathrm{rk}\mathfrak{g}_i$. Then $\mathfrak{h}$
can only differs from $\mathfrak{g}$ for one simple factor. Otherwise, assuming $\mathfrak{h}_i\neq\mathfrak{g}_i$ and the bi-invariant orthogonal decomposition $\mathfrak{g}_i=\mathfrak{h}_i+\mathfrak{m}_i$ for $i=1$ and $2$, then we can take a root $\alpha$ of $\mathfrak{m}_1$ and $\beta$ of $\mathfrak{m}_2$.
Then Condition (A) is not satisfied for $\alpha$ and $\beta$, which is
a contradiction to Lemma \ref{(A)}. To summarize, we have the following
immediate corollary.

\begin{corollary}\label{corollary}
Let $G$ be a compact connected simply connected Lie group and $H$ a closed subgroup of $G$
such that $G$ acts effectively on $G/H$ and the dimension of $G/H$ is even. If $G/H$
admits a positively curved invariant Finsler metric, then $G$ must be simple.
\end{corollary}

Up to local isometries, or under the additional  condition that $\mathfrak{h}$ does not contain a nontrivial ideal of $\mathfrak{g}$, we can reduce
our discussion to the case that $G$ is simple.
Then Proposition 6.1 of \cite{Wallach1972}
provides the complete list of $(\mathfrak{g},\mathfrak{h})$ of simple
compact $\mathfrak{g}$ and subalgebra $\mathfrak{h}$ with $\mathrm{rk}\mathfrak{h}=\mathrm{rk}\mathfrak{g}$ satisfying Condition (A). We list all the possibilities as the following:
$(\mathfrak{a}_n,\mathfrak{a}_{n-1}\oplus\mathbb{R})$,
$(\mathfrak{c}_n,\mathfrak{c}_{n-1}\oplus\mathfrak{a}_1)$,
$(\mathfrak{c}_n,\mathfrak{c}_{n-1}\oplus\mathbb{R})$,
$(\mathfrak{f}_4,\mathfrak{b}_4)$, $(\mathfrak{g}_2,\mathfrak{a}_2)$,
$(\mathfrak{a}_2,\mathbb{R}^2)$,
$(\mathfrak{c}_3,\mathfrak{a}_1\oplus\mathfrak{a}_1\oplus\mathfrak{a}_1)$, and
$(\mathfrak{f}_4,\mathfrak{d}_4)$.

 If we assume further  that the compact Lie group $G$ is connected and simply connected, then we can get the list in Theorem \ref{thm2}. On the other hand,   the existence of positively curved Riemannian homogeneous metrics on the list has been
proved in \cite{Ber61} and \cite{Wallach1972}. This completes the proof of
Theorem \ref{thm1}.

\section{Proof of Theorem \ref{thm2}}
In this section we completes the proof of  Theorem \ref{thm2}. The proof is just a case by case check on the isotropy representation of the connected simply connected homogeneous manifolds admitting
invariant Finsler metrics of positive flag curvature. We will also present some information on the invariant Finsler metrics on these coset spaces, which will be of interest in its own right from the point of view of Finsler geometry. By Theorem \ref{thm1}, the list of connected simply connected homogeneous manifolds admitting
invariant Finsler metrics of positive flag curvature coincides
with that of the Riemannian case. Now by the results of Wallach (see \cite{Wallach1972}), the list of pairs $(G, H)$ for the coset spaces $G/H$ is as 
  the following (we only give the pairs of Lie groups which satisfy the assumptions in Theorem \ref{thm2}):
\begin{description}
\item{(1)} The Riemannian symmetric pairs $(G,H)$ of compact type of rank one;
\item{(2)}  The pair $(\mathrm{Sp}(n), \mathrm{Sp}(n-1)\mathrm{U}(1))$, where $n\geq 2$;
\item{(3)} The pair $(G_2, \mathrm{SU}(3))$;
\item{(4)}  The pair $(\mathrm{SU}(3), T^2)$, where $T^2$ is a maximal torus of $\mathrm{SU}(3)$;
\item{(5)} The pair $(\mathrm{Sp}(3), \mathrm{Sp}(1)\times \mathrm{Sp}(1)\times \mathrm{Sp}(1))$.
\item{(6)} The pair $(\mathrm{F}_4, \mathrm{Spin}(8))$.
\end{description}

Now we give a case by case study of invariant Finsler metrics on coset spaces corresponding to the above coset pairs.

(1)\quad The Riemannian symmetric pairs $(G,H)$ of compact type of rank one. Note that in this case the isotropy subgroup $H$ acts transitively on the unit sphere of $T_o(G/H)$, where $o=eH$ is the origin, with respect to standard Riemannian metric. Then any $G$-invariant Finsler metric on $G/H$ must be a positive multiple of the standard Riemannian metric.

(2)\quad The pair $(\mathrm{Sp}(n), \mathrm{Sp}(n-1)\mathrm{U}(1))$, where $n\geq 2$. A description of the isotropy representation of the coset space is presented in \cite{ZI}.
The tangent space $T_o(G/H)$ can be decomposed as $T_o(G/H)=\mathbb{R}^2\oplus \mathbb{H}^{n-1}$.
The subgroup $\mathrm{Sp}(n-1)$ of $H$ acts trivially on $\mathrm{R}^2$, and the action of $\mathrm{U}(1)$ on $\mathrm{R}^2$ is the standard rotation action. Meanwhile, the action of $H$ on
$\mathbb{H}^{n-1}$ is $(A, z)(v)=A(v)\bar{z}$. Therefore the subspaces $\mathbb{R}^2$ and $\mathbb{H}^{n-1}$ are both invariant under $H$ and the action of $H$ is both transitively on
the unit sphere of $\mathbb{R}^2$ as well as that of $\mathbb{H}^{n-1}$ with respect to the standard metrics. As pointed out by Onis$\hat{\rm c}$ik \cite{ON66}, the coset space $\mathrm{Sp}(n) / \mathrm{Sp}(n-1)\mathrm{U}(1)$ is the complex projective space $\mathbb{C}\mathrm{P}^{2n-1}$. The standard metric of $\mathbb{C}\mathrm{P}^{2n-1}$ corresponds
to the inner product
$$\langle\cdot, \cdot\rangle=\langle\cdot,\cdot\rangle_{\mathbb{R}^2}+\langle \cdot,\cdot\rangle_{\mathbb{H}^{n-1}},$$
where $\langle\cdot,\cdot\rangle_{\mathbb{R}^2}$ (resp. $ \langle \cdot,\cdot\rangle_{\mathbb{H}^{n-1}}$) is the standard inner product on $\mathbb{R}^2$ (resp. $\mathbb{H}^{n-1}$). The general form of an invariant Riemannian metric on $\mathrm{Sp}(n) / \mathrm{Sp}(n-1)\mathrm{U}(1)$ is induced by the inner product on $\mathfrak{m}$ of the form
$$\langle\cdot, \cdot\rangle_{(t_1,t_2)}=t_1\langle\cdot,\cdot\rangle_{\mathbb{R}^2}+t_2\langle \cdot,\cdot\rangle_{\mathbb{H}^{n-1}},$$
where $t_1, t_2$ are positive real numbers. By the continuity, if the ratio ${t_1}/{t_2}$ is close enough to $1$, then the corresponding Riemannian metric has positive curvature.
On the other hand, let $F$ be an invariant Finsler metric on $\mathrm{Sp}(n) / \mathrm{Sp}(n-1)\mathrm{U}(1)$. Denote the corresponding Minkowski norm on $\mathfrak{m}$
also as $F$. Then by the transitivity indicated above, the restriction of $F$ in $\mathbb{R}^2$ (resp. $\mathrm{H}^{n-1}$) must be a positive multiple of $\langle\cdot,\cdot\rangle_{\mathbb{R}^2}$ (resp. $\langle \cdot,\cdot\rangle_{\mathbb{H}^{n-1}}$). Therefore $F$ must be of the form
$F=\sqrt{L(\alpha_1,\alpha_2)}$, where $L$ is a homogeneous  positive smooth function  of degree one satisfying some appropriate conditions, and $\alpha_1, \alpha_2$ are quadratic function on $\mathfrak{m}$ defined by
$$\alpha_1 (X, Y)=\langle X, X\rangle_{\mathbb{R}^2},\, \alpha_2 (X, Y)=\langle Y, Y\rangle_{\mathbb{H}^{n-1}},\quad (X,Y)\in \mathfrak{m}.$$
The corresponding metric is exactly an $(\alpha_1,\alpha_2)$-metric introduced by the first two authors in \cite{XuDeng2014}. By the continuity of the flag curvature on the Finsler metric, it is easily seen that if $F$
is close enough to the standard Riemannian metric, then $F$ has positive flag curvature.
For example, consider a family of Minkowski norm $F_\varepsilon$ on $\mathfrak{m}$ defined by
$$
F_\varepsilon(X,Y)=\sqrt{\alpha_1(X,Y)+\alpha_2(X,Y)+\varepsilon \sqrt{\alpha_1(X,Y)^2+\alpha_2(X,Y)^2}},\quad (X,Y)\in\mathfrak{m},
$$
where $\varepsilon$ is a positive real number. Then the corresponding Finsler metric
must have positive flag curvature when  $\varepsilon$ is small enough.

(3)\quad The pair $(G_2, \mathrm{SU}(3))$. Although this  is not a  symmetric pair,
the isotropic representation is transitive on the unit sphere of the tangent space at the origin
of $S^6=G_2/\mathrm{SU}(3)$. Any homogeneous metric $F$ on it must be the standard Riemannian metric of positive constant sectional curvature.

(4)-(6)\quad N. R. Wallach has shown in \cite{Wallach1972} that, for each
pair $(G,H)$ in these cases, we have an $\mathrm{Ad}(H)$-invariant decomposition $\mathfrak{m}=\mathfrak{m}_1+\mathfrak{m}_2+\mathfrak{m}_3$. The dimension
of each $\mathfrak{m}_i$ is 2 for $(\mathrm{SU}(3),T^2)$, 4 for
$(\mathrm{Sp}(3), \mathrm{Sp}(1)\times \mathrm{Sp}(1)\times \mathrm{Sp}(1))$,
and 8 for $(\mathrm{F}_4, \mathrm{Spin}(8))$. Let $\alpha_i^2$ be the
bi-linear function defined by the
restriction of the bi-invariant inner product to each $\mathfrak{m}_i$,
and naturally viewed as bi-linear functions on $\mathfrak{m}$.
N. R. Wallach has given a triple $(t_1,t_2,t_3)$ (there are infinitely many these triples), such that
the inner product $t_1\alpha_1^2+t_2\alpha_2^2+t_3\alpha_3^2$
on $\mathfrak{m}$ defines a positively curved Riemannian homogeneous metric $F_0$ on $G/H$. In Finsler geometry, we have infinitely many way to perturb
$F_0$ in a non-Riemannian manner. For example the Minkowski norm
$$F_\varepsilon=\sqrt{t_1\alpha_1^2+t_2\alpha_2^2+t_3\alpha_3^2+
\varepsilon\sqrt{\alpha_1^2+\alpha_2^2+\alpha_3^2}},$$
defines a non-Riemannian homogeneous metric $F_\epsilon$ on $G/H$.
When the positive parameter $\varepsilon$ is sufficiently close to 0, it is
positively curved. The metrics $F_\varepsilon$ we constructed here are $(\alpha_1,\alpha_2,\alpha_3)$-metrics (see Subsection 2.1), but there are
much more complicated ways to perturb the homogeneous metric $F_0$, which
can result much more homogeneous Finsler metrics on $G/H$ with positive flag curvature.

The proof of Theorem \ref{thm2} is now completed.

 Finally, let us remark that the assumptions that $G$ is compact connected simply connected, $H$ is connected, and $\mathfrak{h}$ does not contain any nonzero ideal
of $\mathfrak{g}$ does not prevent us from using Theorem \ref{thm2} to describe general
homogeneous Finsler spaces of even dimension admitting positive flag curvature. In fact, changing
$G$ to its simply connected covering group, changing $H$ to its identity component, or
cancelling the common product factor from both does not affect the local geometry of
homogeneous Finsler metric in the consideration. Furthermore,  assuming that $\mathfrak{h}$
does not contain any nonzero ideal of $\mathfrak{g}$, and the $G$-homogeneous Finsler metric $F$ on the even dimensional coset space $G/H$ is positively curved, then $G$ can not have a positive
dimensional center. Otherwise there is nonzero center vector of $\mathfrak{g}$ contained in $\mathfrak{h}$, which is a contradiction to our assumption.
Thus $G$ is semisimple and  has a compact connected simply connected covering group. So any even dimensional positively curved homogeneous space $G/H$
must be locally isometric to one in the list of Theorem \ref{thm2}.

\end{document}